\documentclass[10pt,reqno]{amsart}

\usepackage{amsthm, mathrsfs, amsmath, amstext, amsxtra, amsfonts, dsfont, amssymb}
\usepackage{lmodern}
\usepackage[dvipsnames]{xcolor}
\usepackage[colorlinks, linkcolor=red, citecolor=blue, urlcolor=OliveGreen, pagebackref, hypertexnames=false]{hyperref}

\usepackage[belowskip=2pt,aboveskip=0pt]{caption}
\usepackage{booktabs}
\usepackage{geometry}
\newcommand{\R}{\mathbb R} 
\newcommand{\eps}{\epsilon}

\newcommand{\Fcal}{\mathcal F}

\newcommand{\im}[1]{\mbox{Im} #1} 
\newcommand{\scal}[1]{\left\langle #1 \right\rangle} 

\newcommand{\defendproof}{\hfill $\Box$} 

\newtheorem{theorem}{Theorem}[section]
\newtheorem{lemma}[theorem]{Lemma} 
\newtheorem{proposition}[theorem]{Proposition}
 
\theoremstyle{definition}
\newtheorem{definition}[theorem]{Definition}
\newtheorem{remark}[theorem]{Remark}

\title[Instability standing waves FNLS]{On instability of standing waves for the mass-supercritical fractional nonlinear Schr\"odinger equation} 

\author[V. D. Dinh]{Van Duong Dinh}
\address[V. D. Dinh]{Institut de Math\'ematiques de Toulouse UMR5219, Universit\'e Toulouse CNRS, 31062 Toulouse Cedex 9, France and Department of Mathematics, HCMC University of Pedagogy, 280 An Duong Vuong, Ho Chi Minh, Vietnam}
\email{dinhvan.duong@math.univ-toulouse.fr}

\keywords{Fractional nonlinear Schr\"odinger equation; Standing wave; Instability; Localized virial estimate; Blowup}
\subjclass[2010]{35B35, 35Q55}

\begin{document}

\begin{abstract}
We consider the focusing $L^2$-supercritical fractional nonlinear Schr\"odinger equation
\[
i\partial_t u - (-\Delta)^s u = -|u|^\alpha u, \quad (t,x) \in \R^+ \times \R^d,
\]
where $d\geq 2, \frac{d}{2d-1} \leq s <1$ and $\frac{4s}{d}<\alpha<\frac{4s}{d-2s}$. By means of the localized virial estimate, we prove that the ground state standing wave is strongly unstable by blow-up. This result is a complement to a recent result of Peng-Shi [{\bf J. Math. Phys. 59 (2018), 011508}] where the stability and instability of standing waves were studied in the $L^2$-subcritical and $L^2$-critical cases.
\end{abstract}

\maketitle

\section{Introduction}
\setcounter{equation}{0}

In recent years, there has been a great deal of interest in studying the nonlinear fractional Schr\"odinger equation, namely
\begin{align}
i\partial_t u - (-\Delta)^s u = f(u), \label{general FNLS}
\end{align}
where $0<s<1$ and $f(u)$ is the nonlinearity. The fractional differential operator $(-\Delta)^s$ is defined by $(-\Delta)^s u = \Fcal^{-1} \left[|\xi|^{2s} \Fcal u\right]$, where $\Fcal$ and $\Fcal^{-1}$ are the Fourier transform and inverse Fourier transform, respectively. The fractional nonlinear Schr\"odinger equation was first discovered by N. Laskin \cite{Laskin1, Laskin2} owing to the extension of the Feynman path integral, from the Brownian-like to L\'evy-like quantum mechanical paths. The fractional nonlinear Schr\"odinger equation also appears in the continuum limit of discrete models with long-range interactions (see \cite{KirkLenzStaf}) and in the description of Boson stars as well as in water wave dynamics (see e.g. \cite{FrohJonsLenz} or \cite{IonePusa}). 

In this paper, we consider the Cauchy problem for the focusing fractional nonlinear Schr\"odinger equation
\begin{align}
\left\{
\begin{array}{rcl}
i\partial_t u - (-\Delta)^s u &=& - |u|^{\alpha} u, \quad (t,x) \in \R^+ \times \R^d, \\
u(0) &=& u_0, 
\end{array}
\right.
\label{FNLS}
\end{align}
where $u$ is a complex valued function defined on $\R^+ \times \R^d$, $d\geq 1, 0 \leq s <1$ and $0<\alpha<\alpha^\star$ with
\begin{align}
\alpha^\star:= \left\{
\begin{array}{cl}
\frac{4s}{d-2s} &\text{if } d>2s,\\
\infty &\text{if } d\leq 2s.
\end{array}
\right. \label{define alpha star}
\end{align}
Throughout the sequel, we call a standing wave a solution of $(\ref{FNLS})$ of the form $e^{i\omega t} \phi_\omega$, where $\omega \in \R$ is a frequency and $\phi_\omega \in H^s$ is a non-trivial solution to the elliptic equation 
\begin{align}
(-\Delta)^s \phi_\omega + \omega \phi_\omega - |\phi_\omega|^\alpha \phi_\omega =0. \label{elliptic equation omega}
\end{align}
We are interested in the instability of standing waves for $(\ref{FNLS})$. Before stating our main result, let us recall known results of orbital stability and instability of standing waves for $(\ref{general FNLS})$. In the case of Hartree-type nonlinearity $f(u)=-(|x|^{-\gamma} * |u|^2) u$, Wu \cite{Wu} showed the existence of stable standing waves for $d\geq 1$, $0<s<1$ and $0<\gamma<\min\{2s,d\}$. Zhang-Zhu \cite{ZhangZhu} extended the result of Wu and showed the strong instability of standing waves for $d\geq 2$, $0<s<1$ and $\gamma=2s$. Recently, Cho-Fall-Hajaiej-Markowich-Trabelsi \cite{ChoFallHajaiejMarkowichTrabelsi} studied the existence of stable standing waves for more general Hartree-type nonlinearities. In the case of Choquard nonlinearity $f(u)=-(I_\beta * |u|^p)|u|^{p-2}u$ with $d\geq 3$, $0<s<1$, $1+\frac{\beta}{d}<p<1+\frac{2s+\beta}{d}$ and
\[
I_\beta:= A(\beta) |x|^{-(d-\beta)}, \quad A(\beta) = \frac{\Gamma\left(\frac{d-\beta}{2} \right)}{\Gamma\left(\frac{\beta}{2}\right) \pi^{\frac{d}{2}} 2^\beta}, \quad 0<\beta<d,
\]
Feng-Zhang \cite{FengZhang-choquard} established the stability of standing waves under an assumption on the local well-posedness of the Cauchy problem $(\ref{general FNLS})$. In the case of combined power-type and Choquard nonlinearities $f(u)= - |u|^\alpha u - (I_\beta*|u|^p) |u|^{p-2}u$, Bhattarai \cite{Bhattarai} proved the existence of stable standing waves for $d\geq 2$, $0<s<1$, $0<\alpha<\frac{4s}{d}$ and $2 \leq p< 1+\frac{2s+\beta}{d}$. Recently, Feng-Zhang \cite{FengZhang-combined} showed the stability of standing waves for $d\geq 2$, $0<s<1$, $\alpha=\frac{4s}{d}$ and $1+\frac{\beta}{d}<p<1+\frac{2s+\beta}{d}$ under an assumption on the local theory of the Cauchy problem $(\ref{general FNLS})$ and an assumption on the initial data $\|u_0\|_{L^2}< \|Q\|_{L^2}$, where $Q$ is the ground state of
\begin{align}
(-\Delta)^s Q + Q - |Q|^{\frac{4s}{d}} Q =0. \label{elliptic equation introduction}
\end{align}
In the case of combined power-type nonlinearities $f(u)=-|u|^{\alpha_1}u - |u|^{\alpha_2}u$, Guo-Huang \cite{GuoHuang} showed the existence of stable standing waves for $0<\alpha_1<\alpha_2<\frac{4s}{d}$. Cho-Hwang-Hajaiej-Ozawa \cite{ChoHwangHajaiejOzawa} proved the stability of standing waves for more general subcritical nonlinearities. For $0<\alpha_1<\alpha_2=\frac{4s}{d}$, Zhu \cite{Zhu-combined}  showed the existence of stable standing waves with $\|u_0\|_{L^2}<\|Q\|_{L^2}$, where $Q$ is the ground state of $(\ref{elliptic equation introduction})$. In the case of logarithmic nonlinearity $f(u)=-u \log (|u|^2)$, Ardila \cite{Ardila} proved the existence of stable standing waves for $d\geq 2$ and $0<s<1$. 

In the case of a single power-type nonlinearity $(\ref{FNLS})$, Peng-Shi \cite{PengShi} recently established the existence of stable standing waves for $d\geq 2$, $0<s<1$ and $0<\alpha<\frac{4s}{d}$. They also proved the strong instablity of standing waves for $d\geq 2$, $\frac{1}{2}<s<1$ and $\alpha=\frac{4s}{d}$. Note that the local well-posedness is not considered in \cite{PengShi}. Due to loss of regularity in Strichartz estimates for the unitary group $e^{-it(-\Delta)^s}$, there are restrictions on the local well-posedness in $H^s$ for $(\ref{FNLS})$ with non-radial initial data (see e.g. \cite{HongSire} or \cite{Dinh-fract}). One can overcome the loss of derivatives in Strichartz estimates by considering radial initial data. However, it leads to another restriction on the validity of $d$ and $s$, that is, $d\geq 2$ and $\frac{d}{2d-1} \leq s<1$. We refer the reader to \cite{Dinh-mfnls} for the local well-posedness for $(\ref{FNLS})$ with $H^s$ radial initial data.

The main purpose of this paper is to show the strong instablity of ground state standing waves for $(\ref{FNLS})$ in the mass-supercritical and energy-subcritical case $\frac{4s}{d}<\alpha<\frac{4s}{d-2s}$. In order to state our main result, let us introduce the notion of ground states related to $(\ref{elliptic equation omega})$. 
\begin{definition} [Ground states] 
	A non-zero, non-negative $H^s$ solution $\phi_\omega$ to $(\ref{elliptic equation omega})$ is called a {\bf ground state related to $(\ref{elliptic equation omega})$} if it is a minimizer of the Weinstein's functional
	\begin{align}
	J(v):= \left[ \|v\|^{\frac{d\alpha}{2s}}_{\dot{H}^s} \|v\|^{\alpha+2-\frac{d\alpha}{2s}}_{L^2} \right] \div \|v\|^{\alpha+2}_{L^{\alpha+2}}, \label{weinstein functional}
	\end{align}
	that is, 
	\[
	J(\phi_\omega) = \inf \left\{ J(v) \ : \ v \in H^s \backslash \{0\} \right\}.
	\]
\end{definition}
Similarly, a non-zero, non-negative $H^s$ solution $\phi$ to the elliptic equation
\begin{align}
(-\Delta)^s \phi + \phi - |\phi|^\alpha \phi =0 \label{elliptic equation}
\end{align}
is called a ground state related to $(\ref{elliptic equation})$ if it is a minimizer of the Weinstein's functional $(\ref{weinstein functional})$. Note that for $d\geq 1$, $0<s<1$ and $0<\alpha<\alpha^\star$, the existence and uniqueness (up to symmetries) of ground states related to $(\ref{elliptic equation})$ were established recently in \cite{FrankLenzmann, FrankLenzmannSilvestre}. Moreover, the ground state related to $(\ref{elliptic equation})$ can be choosen to be radially symmetric, strictly positive and strictly decreasing in $|x|$. 

Now, let $\phi$ be the ground state related to $(\ref{elliptic equation})$. It is easy to see that for $\omega>0$, the scaling $\phi_\omega(x):= \omega^{\frac{1}{\alpha}} \phi \left(\omega^{\frac{1}{2s}} x\right)$ maps a solution of $(\ref{elliptic equation})$ to a solution of $(\ref{elliptic equation omega})$. We thus get a non-zero, non-negative $H^s$ solution to $(\ref{elliptic equation omega})$. On the other hand, a direct computation shows that
\begin{align*}
\|\phi_\omega\|^2_{L^2} = \omega^{\frac{2s}{\alpha}-\frac{d}{2}} \|\phi\|^2_{L^2}, \quad 
\|\phi_\omega\|^2_{\dot{H}^s} = \omega^{\frac{2s}{\alpha} +s -\frac{d}{2}} \|\phi\|^2_{\dot{H}^s}, \quad 
\|\phi_\omega\|^{\alpha+2}_{L^{\alpha+2}} = \omega^{\frac{s}{\alpha}(\alpha+2) - \frac{d}{2}} \|\phi\|^{\alpha+2}_{L^{\alpha+2}}.
\end{align*}
It follows that
\[
J(\phi_\omega) = \left[ \|\phi_\omega\|^{\frac{d\alpha}{2s}}_{\dot{H}^s} \|\phi_\omega\|^{\alpha+2-\frac{d\alpha}{2s}}_{L^2} \right] \div \|\phi_\omega\|^{\alpha+2}_{L^{\alpha+2}} = J(\phi).
\]
Thus, $\phi_\omega$ is a minimizer of the Weinstein's functional. By definition, $\phi_\omega$ is a ground state related to $(\ref{elliptic equation omega})$. Similarly, we can construct ground states related to $(\ref{elliptic equation})$ from ground states related to $(\ref{elliptic equation omega})$. This implies the existence and uniqueness (up to symmetries) of ground states related to $(\ref{elliptic equation omega})$ when $\omega>0$. Moreover, the unique (up to symmetries) ground state related to $(\ref{elliptic equation omega})$ can be choosen to be radially symmetric, strictly positive and strictly decreasing in $|x|$. 

It is easy to see that $(\ref{elliptic equation omega})$ can be written as $S'_\omega(\phi_\omega)=0$, where
\begin{align*}
S_\omega(v) &:= E(v)+ \frac{\omega}{2} \|v\|^2_{L^2} \\
&\mathrel{\phantom{:}}= \frac{1}{2} \|v\|^2_{\dot{H}^s} + \frac{\omega}{2} \|v\|^2_{L^2} -\frac{1}{\alpha+2} \|v\|^{\alpha+2}_{L^{\alpha+2}}
\end{align*}
is the action functional. We also define the Nehari functional 
\[
K_\omega(v):= \left.\partial_\lambda S_\omega(\lambda v) \right|_{\lambda=1} = \|v\|^2_{\dot{H}^s} + \omega \|v\|^2_{L^2} - \|v\|^{\alpha+2}_{L^{\alpha+2}}.
\]
From now on, we denote the functional
\[
H_\omega(v):= \|v\|^2_{\dot{H}^s} + \omega \|v\|^2_{L^2}. 
\]
It is easy to see that for $\omega>0$ fixed, 
\begin{align}
H_\omega(v) \sim \|v\|^2_{H^s}. \label{equivalent norm}
\end{align}
Let us start with the following observation concerning the ground state related to $(\ref{elliptic equation omega})$.
\begin{proposition} \label{proposition characterization ground state}
Let $d\geq 1$, $0< s <1$, $0<\alpha<\alpha^\star$ with $\alpha^\star$ as in $(\ref{define alpha star})$ and $\omega>0$. Let $\phi_\omega$ be the ground state related to $(\ref{elliptic equation omega})$. Then
\begin{align}
S_\omega(\phi_\omega) = \inf \left\{ S_\omega(v) \ : \ v \in H^s\backslash \{0\}, \ K_\omega(v)=0 \right\}. \label{property S_omega phi_omega}
\end{align}
\end{proposition}
We refer the reader to Section $\ref{section characterization ground state}$ for the proof of the above result. We next recall the definition of the strong instability of standing waves.
\begin{definition} [Strong instability] \label{definition strong instability}
A standing wave $e^{i\omega t} \phi_\omega$ is strongly unstable if for any $\eps>0$, there exists $u_0 \in H^s$ such that $\|u_0 - \phi_\omega\|_{H^s} <\eps$ and the solution $u(t)$ to $(\ref{FNLS})$ with initial data $u_0$ blows up in finite time.
\end{definition}
Our main result in this paper reads as follows.
\begin{theorem} \label{theorem strong instability}
Let $d\geq 2$, $\frac{d}{2d-1} \leq s <1$, $\frac{4s}{d} <\alpha< \frac{4s}{d-2s}$, $\alpha<4s$, $\omega>0$ and $\phi_\omega$ be the ground state related to $(\ref{elliptic equation omega})$. Then the ground state standing wave $e^{i\omega t} \phi_\omega$ is strongly unstable.
\end{theorem}
Note that the condition $\alpha<4s$ is technical due to the localized virial estimate (see Remark $\ref{lemma localized virial estimate}$). However, this only leads to a restriction in the two dimensional case, that is, $\frac{2}{3} \leq s <1$ and $2s <\alpha <4s$. 

To our knowledge, the usual strategy to show the strong instability of standing waves is to use the variational characterization of the ground states as minimizers of the action functional and the virial identity, namely
\[
\frac{d}{dt} \|x u(t)\|^2_{L^2} = 4 \im{\left( \int \overline{u}(t) x\cdot \nabla u(t) dx \right)}.
\]
However, in the case $0<s<1$, there is no such virial identity. To overcome this difficulty, we use localized virial estimates for radial solutions of $(\ref{FNLS})$. These localized virial estimates were proved by Boulenger-Himmelsbach-Lenzmann \cite{BoulengerHimmelsbachLenzmann} to show the existence of radial blow-up solutions for $(\ref{FNLS})$ in the mass-critical and mass-supercritical cases. As far as we know, this paper seems to be the first one dealing with the strong instability of standing waves for the fractional nonlinear Schr\"odinger equation with power-type nonlinearity in the mass-supercritical and energy-subcritical regime. The method used to prove Theorem $\ref{theorem strong instability}$ is robust, and can be applied to show the strong instability of standing waves for the fractional nonlinear Schr\"odinger equation with other type of nonlinearities, such as Hartree, Choquard, combined power-type,... A similar approach based on localized virial estimates is used in \cite{BensouilahDinh} to show the strong instability of radial standing waves for the nonlinear Schr\"odinger equation with inverse-square potential in the mass-supercritical and energy-subcritical case. 

The paper is oganized as follows. In Section $\ref{section preliminary}$, we recall basic tools needed in the sequel such as the sharp Gagliardo-Nirenberg inequality, the Pohozaev's identities and the profile decomposition. In Section $\ref{section characterization ground state}$, we show the characterization of the ground state related to $(\ref{elliptic equation omega})$ given in Proposition $\ref{proposition characterization ground state}$. We will give the proof of the strong instablity of the ground state standing waves given in Theorem $\ref{theorem strong instability}$ in Section $\ref{section strong instability}$.

\section{Preliminaries} \label{section preliminary}
\setcounter{equation}{0}
Let us recall some basic tools related to $(\ref{FNLS})$ which are needed in this paper. Let us start with the sharp Gagliardo-Nirenberg inequality. 
\begin{lemma}[Sharp Gagliardo-Nirenberg inequality \cite{BoulengerHimmelsbachLenzmann, FrankLenzmann, FrankLenzmannSilvestre}] \label{lemma gagliardo-nirenberg inequality}
Let $d\geq 1$, $0<s<1$, $0<\alpha <\alpha^\star$ with $\alpha^\star$ as in $(\ref{define alpha star})$. Then for any $u \in H^s$,
\begin{align}
\|u\|^{\alpha+2}_{L^{\alpha+2}} \leq C_{\emph{opt}} \|u\|^{\frac{d\alpha}{2s}}_{\dot{H}^s} \|u\|_{L^2}^{\alpha+2-\frac{d\alpha}{2s}}, \label{sharp gagliardo-nirenberg inequality}
\end{align}
where the optimal constant $C_{\emph{opt}}$ is given by
\[
C_{\text{opt}} = \left(\frac{2s(\alpha+2)-d\alpha}{d\alpha} \right)^{\frac{d\alpha}{4s}} \frac{2s(\alpha+2)}{2s (\alpha+2) - d\alpha} \frac{1}{\|Q\|_{L^2}^\alpha}.
\]
Here $Q$ is the unique (up to symmetries) positive radial solution to the elliptic equation
\[
(-\Delta)^s Q + Q - |Q|^\alpha Q=0. 
\]
Moreover, the following Pohozaev's identities hold true:
\begin{align}
\|Q\|^2_{L^2} = \frac{4s-(d-2s)\alpha}{d\alpha} \|Q\|^2_{\dot{H}^s} = \frac{4s-(d-2s)\alpha}{2s(\alpha+2)} \|Q\|^{\alpha+2}_{L^{\alpha+2}}. \label{pohozaev identities elliptic equation}
\end{align}
\end{lemma}
We refer the reader to \cite{FrankLenzmann}, \cite[Proposition 3.1, Theorem 3.4]{FrankLenzmannSilvestre} and \cite[Appendix]{BoulengerHimmelsbachLenzmann} for the proof of the above result. In the same spirit as $(\ref{pohozaev identities elliptic equation})$, we have the following Pohozaev's identities associated to $(\ref{elliptic equation omega})$:
\begin{align}
\omega \|\phi_\omega\|^2_{L^2} = \frac{4s-(d-2s)\alpha}{d\alpha} \|\phi_\omega\|^2_{\dot{H}^s} = \frac{4s-(d-2s)\alpha}{2s(\alpha+2)} \|\phi_\omega\|^{\alpha+2}_{L^{\alpha+2}}. \label{pohozaev identities elliptic equation omega}
\end{align}
	
We next recall the profile decomposition of bounded $H^s$ sequences. 
\begin{lemma} [Profile decomposition \cite{Dinh-mfnls, Zhu}] \label{lemma profile decomposition}
Let $d\geq 1$ and $0<s<1$. Let $(v_n)_{n\geq 1}$ be a bounded sequence in $H^s$. Then there exist a subsequence of $(v_n)_{n\geq 1}$ (still denoted $(v_n)_{n\geq 1}$), a family $(x^j_n)_{n\geq 1}$ of sequences in $\R^d$ and a sequence $(V^j)_{j\geq 1}$ of $H^s$ functions such that
\begin{itemize}
	\item for every $k \ne j$, 
	\begin{align}
	|x^k_n - x^j_n| \rightarrow \infty, \label{orthogonality}
	\end{align} 
	as $n\rightarrow \infty$;
	\item for every $l \geq 1$ and every $x \in \R^d$,
	\begin{align}
	v_n(x) = \sum_{j=1}^l V^j(x-x^j_n) + v^l_n(x), \label{profile decomposition}
	\end{align}
	with 
	\begin{align}
	\limsup_{n\rightarrow \infty} \|v^l_n\|_{L^q} \rightarrow 0, \label{profile error}
	\end{align}
	as $l \rightarrow \infty$ for every $q \in (2, 2+ \alpha^\star)$ with $\alpha^\star$ as in $(\ref{define alpha star})$. 
	\end{itemize}
Moreover, for every $l\geq 1$, the following expansions hold true:
\begin{align}
\|v_n\|^2_{L^2} &= \sum_{j=1}^l \|V^j\|^2_{L^2} + \|v^l_n\|^2_{L^2} + o_n(1), \label{L2 expansion} \\
\|v_n\|^2_{\dot{H}^s} &= \sum_{j=1}^l \|V^j\|^2_{\dot{H}^s} + \|v^l_n\|^2_{\dot{H}^s} + o_n(1), \label{dot Hs expansion} \\
\|v_n\|^{\alpha+2}_{L^{\alpha+2}} &= \sum_{j=1}^l \|V^j\|^{\alpha+2}_{L^{\alpha+2}} + \|v^l_n\|^{\alpha+2}_{L^{\alpha+2}} + o_n(1), \label{lebesgue norm expansion}
\end{align}
as $n\rightarrow \infty$.
\end{lemma}
We refer the reader to \cite[Proposition 2.3]{Zhu} or \cite[Theorem 3.1]{Dinh-mfnls} for the proof of this result which is similar to the one proved by Hmidi-Keraani \cite[Proposition 3.1]{HmidiKeraani}. 
\begin{remark}
The number of non-zero terms in $(V^j)_{j\geq 1}$ may be one, finite or infinite, which may correspond to compactness, dichotomy, and vanishing, respectively, in the concentration compactness principle proposed by Lions \cite{Lions}. The profile decomposition given in Lemma $\ref{lemma profile decomposition}$ may look as another equivalent description of the concentration-compactness principle. However, there are two major advantages: one is that we can inject the decomposing expression $(\ref{profile decomposition})$ into our aim functionals, and the other is that the decomposition is orthogonal by $(\ref{orthogonality})$ and norms of $(v_n)_{n\geq 1}$ have similar decompositions $(\ref{L2 expansion})-(\ref{lebesgue norm expansion})$. These properties are useful in the calculus of variational methods.		
\end{remark}

\section{Characterization of the ground state} \label{section characterization ground state}
\setcounter{equation}{0}
In this section, we give the proof of the characterization of the ground state given in Proposition $\ref{proposition characterization ground state}$.

\noindent {\it Proof of Proposition $\ref{proposition characterization ground state}$.} The proof is done by several steps. 
	
\noindent {\bf Step 1.} We first show that the minimizing problem 
\begin{align}
d(\omega):=\inf \left\{ S_\omega(v) \ : \ v \in H^s\backslash \{0\}, \ K_\omega(v)=0 \right\} \label{minimizing problem}
\end{align}
is attained and $d(\omega)>0$. The later fact is easy to see. Indeed, let $v \in H^s \backslash \{0\}$ be such that $K_\omega(v)=0$. By the Sobolev embedding, the equivalent norm $(\ref{equivalent norm})$ and the fact $H_\omega(v) = \|v\|^{\alpha+2}_{L^{\alpha+2}}$, we have
\[
\|v\|^2_{L^{\alpha+2}} \leq C_1 \|v\|^2_{H^s} \leq C_2 H_\omega(v) = C_2 \|v\|^{\alpha+2}_{L^{\alpha+2}},
\]
for some $C_1, C_2>0$. This implies that
\[
S_\omega(v) = \frac{1}{2} H_\omega(v) - \frac{1}{\alpha+2} \|v\|^{\alpha+2}_{L^{\alpha+2}} = \frac{\alpha}{2(\alpha+2)} \|v\|^{\alpha+2}_{L^{\alpha+2}}  \geq \frac{\alpha}{2(\alpha+2)} \left(\frac{1}{C_2} \right)^{\frac{\alpha+2}{\alpha}}.
\]
Taking the infimum over $v$, we get $d(\omega) >0$. We now show the minimizing problem $(\ref{minimizing problem})$ is attained. Indeed, let $(v_n)_{n\geq 1}$ be a minimizing sequence of $d(\omega)$, i.e. $v_n \in H^s\backslash \{0\}, K_\omega(v_n) =0$ and $S_\omega(v_n) \rightarrow d(\omega)$ as $n\rightarrow \infty$. Since $K_\omega(v_n) =0$, we have $H_\omega(v_n)= \|v_n\|^{\alpha+2}_{L^{\alpha+2}}$ for any $n\geq 1$. The fact $S_\omega(v_n) \rightarrow d(\omega)$ as $n\rightarrow \infty$ implies that
\[
\frac{\alpha}{2(\alpha+2)} \|v_n\|^{\alpha+2}_{L^{\alpha+2}} = \frac{\alpha}{2(\alpha+2)} H_\omega(v_n) \rightarrow d(\omega),
\]
as $n\rightarrow \infty$. We infer that there exists $C>0$ such that
\[
H_\omega(v_n) \leq \frac{2(\alpha+2)}{\alpha} d(\omega) + C,
\]
for all $n\geq 1$. By $(\ref{equivalent norm})$, it follows that $(v_n)_{n\geq 1}$ is a bounded sequence in $H^s$. Thanks to the profile decomposition given in Lemma $\ref{lemma profile decomposition}$, there exist a subsequence still denoted by $(v_n)_{n\geq 1}$, a family $(x^j_n)_{n\geq 1}$ of sequences in $\R^d$ and a sequence $(V^j)_{j\geq 1}$ of $H^s$-functions such that for every $l\geq 1$ and every $x \in \R^d$,
\[
v_n(x) = \sum_{j=1}^l V^j(x-x^j_n) + v^l_n(x),
\]
and $(\ref{profile error})-(\ref{lebesgue norm expansion})$ hold. We have from $(\ref{L2 expansion})$ and $(\ref{dot Hs expansion})$ that 
\[
H_\omega(v_n) = \sum_{j=1}^l H_\omega (V^j) + H_\omega(v^l_n) + o_n(1),
\]
as $n\rightarrow \infty$. This implies that
\begin{align*}
K_\omega(v_n) &= H_\omega(v_n) - \|v_n\|^{\alpha+2}_{L^{\alpha+2}} \\
&= \sum_{j=1}^l H_\omega(V^j) + H_\omega(v^l_n) - \|v_n\|^{\alpha+2}_{L^{\alpha+2}} + o_n(1) \\
&= \sum_{j=1}^l K_\omega(V^j) + \sum_{j=1}^l \|V^j\|^{\alpha+2}_{L^{\alpha+2}} - \|v_n\|^{\alpha+2}_{L^{\alpha+2}} + H_\omega(v^l_n) + o_n(1).
\end{align*}
Since $K_\omega(v_n) =0$, $\|v_n\|^{\alpha+2}_{L^{\alpha+2}} \rightarrow \frac{2(\alpha+2)}{\alpha} d(\omega)$ as $n\rightarrow \infty$ and $H_\omega(v^l_n) \geq 0$ for all $n\geq 1$, we infer that
\begin{align}
\sum_{j=1}^l K_\omega(V^j) + \sum_{j=1}^l \|V^j\|^{\alpha+2}_{L^{\alpha+2}} - \frac{2(\alpha+2)}{\alpha} d(\omega) \leq 0, \label{limit estimate 1}
\end{align}
or equivalently,
\begin{align}
\sum_{j=1}^l H_\omega(V^j) - \frac{2(\alpha+2)}{\alpha} d(\omega) \leq 0. \label{limit estimate 2}
\end{align}
On the other hand, by $(\ref{profile error})$ and $(\ref{lebesgue norm expansion})$, we have
\begin{align}
\frac{2(\alpha+2)}{\alpha} d(\omega) = \lim_{n\rightarrow \infty} \|v_n\|^{\alpha+2}_{L^{\alpha+2}} = \sum_{j=1}^\infty \|V^j\|^{\alpha+2}_{L^{\alpha+2}}. \label{limit estimate 3}
\end{align}
Combining $(\ref{limit estimate 1})$, $(\ref{limit estimate 2})$ and $(\ref{limit estimate 3})$, we obtain
\begin{align}
\sum_{j=1}^\infty K_\omega(V^j) \leq 0, \quad \sum_{j=1}^\infty H_\omega(V^j) \leq \frac{2(\alpha+2)}{\alpha} d(\omega). \label{limit estimate 4}
\end{align}
We claim that $K_\omega(V^j) =0$ for all $j \geq 1$. Indeed, suppose that there exists $j_0 \geq 1$ such that $K_\omega(V^{j_0}) <0$. Set 
\[
\lambda_0:= \left(\frac{H_\omega(V^{j_0})}{ \|V^{j_0}\|^{\alpha+2}_{L^{\alpha+2}}} \right)^{\frac{1}{\alpha}}.
\]
Since $K_\omega(V^{j_0}) <0$, we see that $\lambda_0 \in (0,1)$. Moreover, for $\lambda>0$, we have
\[
K_\omega(\lambda V^{j_0}) = \lambda^2 H_\omega(V^{j_0}) - \lambda^{\alpha+2} \|V^{j_0}\|^{\alpha+2}_{L^{\alpha+2}}.
\]
By the choice of $\lambda_0$, we see that $K_\omega(\lambda_0 V^{j_0})=0$. Thus,
\[
d(\omega) \leq S_\omega(\lambda_0 V^{j_0}) = \frac{\alpha}{2(\alpha+2)} H_\omega(\lambda_0 V^{j_0}) =  \frac{\alpha \lambda_0^2}{2(\alpha+2)} H_\omega(V^{j_0}) < \frac{\alpha}{2(\alpha+2)} H_\omega(V^{j_0}). 
\]
Thanks to the second inequality of $(\ref{limit estimate 4})$, we get
\[
d(\omega) <\frac{\alpha}{2(\alpha+2)} H_\omega(V^{j_0}) \leq d(\omega),
\]
which is absurb. We next claim that there exists exactly one $j$ such that $V^j$ is non-zero. Indeed, if there exists $V^{j_1}$ and $V^{j_2}$ non-zero, then the second inequality of $(\ref{limit estimate 4})$ implies that both $H_\omega(V^{j_1})$ and $H_\omega(V^{j_2})$ are strictly smaller than $\frac{2(\alpha+2)}{\alpha} d(\omega)$. However, since $K_\omega(V^{j_1}) =0$, we learn from the definition of $d(\omega)$ that 
\[
\frac{2(\alpha+2)}{\alpha} d(\omega) \leq \frac{2(\alpha+2)}{\alpha} S_\omega(V^{j_1}) = H_\omega(V^{j_1}) <\frac{2(\alpha+2)}{\alpha} d(\omega),
\]
which is again absurd. 
	
Without loss of generality, we may assume that the only non-zero profile is $V^1$. We have from $(\ref{limit estimate 3})$ that
\[
\|V^1\|^{\alpha+2}_{L^{\alpha+2}} = \frac{2(\alpha+2)}{\alpha} d(\omega),
\]
which implies $V^1 \ne 0$. On the other hand, by the first estimate of $(\ref{limit estimate 4})$, we have $K_\omega(V^1) \leq 0$. By the same argument as above, we get $K_\omega(V^1) =0$. Therefore, $V^1$ is a minimizer of $d(\omega)$, and the minimizing problem $(\ref{minimizing problem})$ is attained.
	
\noindent {\bf Step 2.} We will show that $V^1$ is a ground state related to $(\ref{elliptic equation omega})$. Since $V^1$ is a minimizer of $d(\omega)$, there exists a Lagrange multiplier $\mu\in \R$ such that $S'_\omega(V^1) = \mu K'_\omega(V^1)$. We thus have
\[
0 = K_\omega(V^1) = \scal{S'_\omega(V^1), V^1} = \mu \scal{K'_\omega(V^1), V^1}.
\]
It is easy to see that
\[
K'_\omega(V^1) = 2(-\Delta)^s V^1 + 2 \omega V^1 - (\alpha+2) \|V^1|^\alpha V^1.
\]
Hence,
\[
\scal{K'_\omega(V^1), V^1} = 2 H_\omega(V^1) - (\alpha+2) \|V^1\|^{\alpha+2}_{L^{\alpha+2}} = -\alpha \|V^1\|^{\alpha+2}_{L^{\alpha+2}} <0.
\]
It follows that $\mu=0$ and $S'_\omega(V^1) =0$. In particular, $V^1$ is a solution to $(\ref{elliptic equation omega})$. Note that since $S_\omega(|V^1|) = S_\omega(V^1)$ and $K_\omega(|V^1|) = K_\omega(V^1)$, we can choose $V^1$ to be non-negative. 
	
We will show that $V^1$ is a minimizer of the Weinstein'functional $(\ref{weinstein functional})$. Let $v \in H^s \backslash \{0\}$. It is easy to see that $K_\omega(\lambda_0 v)=0$, where
\begin{align}
\lambda_0 := \left( \frac{H_\omega(v)}{\|v\|^{\alpha+2}_{L^{\alpha+2}}} \right)^{\frac{1}{\alpha}} >0. \label{define lambda_0}
\end{align}
By the definition of $d(\omega)$, we have 
\begin{align}
S_\omega(V^1) \leq S_\omega(\lambda_0 v). \label{property V1}
\end{align}  
On the other hand, we have
\[
S_\omega(\lambda v) = \frac{\lambda^2}{2} H_\omega(v) - \frac{\lambda^{\alpha+2}}{\alpha+2} \|v\|^{\alpha+2}_{L^{\alpha+2}},
\]
and
\[
S'_\omega(\lambda v) = \lambda H_\omega(v) - \lambda^{\alpha+1} \|v\|^{\alpha+2}_{L^{\alpha+2}}.
\]
Hence $S'_\omega(\lambda_0 v) =0$, or $\lambda_0 v$ is a solution to $(\ref{elliptic equation omega})$. It follows that both $V^1$ and $\lambda_0 v$ satisfy the following Pohozaev's identities (see e.g. \cite[Appendix]{BoulengerHimmelsbachLenzmann}):
\begin{align*}
\omega \|V^1\|^2_{L^2} &= \frac{4s-(d-2s)\alpha}{2s(\alpha+2)} \|V^1\|^{\alpha+2}_{L^{\alpha+2}} = \frac{4s- (d-2s)\alpha}{d\alpha} \|V^1\|^2_{\dot{H}^s}, \\
\omega \|\lambda_0 v\|^2_{L^2} &= \frac{4s-(d-2s)\alpha}{2s(\alpha+2)} \|\lambda_0 v\|^{\alpha+2}_{L^{\alpha+2}} = \frac{4s- (d-2s)\alpha}{d\alpha} \|\lambda_0 v\|^2_{\dot{H}^s}.
\end{align*}
On one hand, by $(\ref{property V1})$ and the fact $K_\omega(V^1) = K_\omega(\lambda_0 v) =0$, we get
\[
\|V^1\|^{\alpha+2}_{L^{\alpha+2}} \leq \|\lambda_0 v\|^{\alpha+2}_{L^{\alpha+2}}.
\]
On the other hand, using Pohozaev's identities, we have
\begin{align*}
J(v) = J(\lambda_0 v) &= \left[ \|\lambda_0 v\|^{\frac{d\alpha}{2s}}_{\dot{H}^s} \|\lambda_0 v\|^{\alpha+2-\frac{d\alpha}{2s}}_{L^2} \right] \div \|\lambda_0 v\|^{\alpha+2}_{L^{\alpha+2}} \\
&= \left( \frac{d\alpha}{2s(\alpha+2)} \right)^{\frac{d\alpha}{4s}} \left(\frac{4s-(d-2s)\alpha}{2s(\alpha+2)\omega} \right)^{\frac{\alpha+2}{2} - \frac{d\alpha}{4s}} \left(\|\lambda_0 v\|^{\alpha+2}_{L^{\alpha+2}}\right)^{\frac{\alpha}{2}} \\
&\geq \left( \frac{d\alpha}{2s(\alpha+2)} \right)^{\frac{d\alpha}{4s}} \left(\frac{4s-(d-2s)\alpha}{2s(\alpha+2)\omega} \right)^{\frac{\alpha+2}{2} - \frac{d\alpha}{4s}} \left(\|V^1\|^{\alpha+2}_{L^{\alpha+2}}\right)^{\frac{\alpha}{2}} \\
&= J(V^1).
\end{align*}
This implies that $J(V^1) \leq J(v)$ for any $v \in H^s\backslash \{0\}$, or $V^1$ is a minimizer of the Weinstein functional $(\ref{weinstein functional})$. Therefore, $V^1$ is a ground state related to $(\ref{elliptic equation omega})$. 
	
\noindent {\bf Step 3.} Conclusion. By the uniqueness (up to symmetries) of ground states related to $(\ref{elliptic equation omega})$, we obtain $V^1=\phi_\omega$ (up to symmetries), hence $S_\omega(V^1) = S_\omega(\phi_\omega)$. This proves $(\ref{property S_omega phi_omega})$ and the proof is complete.
\defendproof

\section{Strong instability of standing waves} \label{section strong instability}
\setcounter{equation}{0}
The main purpose of this section is to give the proof of Theorem $\ref{theorem strong instability}$. Let us start with the local well-posedness of $(\ref{FNLS})$.

The local well-posedness for $(\ref{FNLS})$ in the energy space $H^s$ was first studied by Hong-Sire in \cite{HongSire} (see also \cite{Dinh-fract}). The proof is based on Strichartz estimates and the contraction mapping argument. Note that for non-radial data, Strichartz estimates have a loss of derivatives. Fortunately, this loss of derivatives can be conpensated for by using Sobolev embedding. However, it leads to a weak local well-posedness in the energy space compared to the well-known nonlinear Schr\"odinger equation ($s=1$). We refer the reader to \cite{HongSire, Dinh-fract} for more details. One can remove the loss of derivatives in Strichartz estimates by considering radially symmetric data. However, it needs a restriction on the validity of $s$, namely $\frac{d}{2d-1} \leq s <1$. More precisely, we have the following local well-posedness for $(\ref{FNLS})$ with radial $H^s$ initial data. 
\begin{proposition} [Local well-posedness \cite{Dinh-mfnls}] \label{proposition local well-posedness}
	Let $d\geq 2$, $\frac{d}{2d-1} \leq s <1$ and $0 <\alpha<\frac{4s}{d-2s}$. Let 
	\[
	p= \frac{4s(\alpha+2)}{\alpha(d-2s)}, \quad q= \frac{d(\alpha+2)}{d+\alpha s}.
	\]
	Then for any $u_0 \in H^s$ radial, there exist $T \in (0,+\infty]$ and a unique solution to $(\ref{FNLS})$ satisfying
	\[
	u \in C([0,T), H^s) \cap L^p_{\emph{loc}} ([0,T), W^{s,q}).
	\]
	Moreover, the following properties hold:
	\begin{itemize}
		\item If $T<+\infty$, then $\|u(t)\|_{\dot{H}^s} \rightarrow +\infty$ as $t \uparrow T$;
		\item $u \in L^a_{\emph{loc}}([0,T), W^{s,b})$ for any $(a,b)$ fractional admissible pair, i.e. 
		\[
		a \in [2,\infty], \quad b \in [2,\infty), \quad (a,b) \ne \left( 2, \frac{4d-2}{2d-3} \right), \quad \frac{2s}{a} + \frac{d}{b} = \frac{d}{2}. 
		\]
		\item There is conservation of mass and energy, 
		\begin{align*}
		\emph{(Mass)} \quad M(u(t)) &= \int |u(t,x)|^2 dx = M(u_0), \\
		\emph{(Energy)} \quad E(u(t)) &= \frac{1}{2} \int |(-\Delta)^{s/2} u(t,x)|^2 dx - \frac{1}{\alpha+2} \int |u(t,x)|^{\alpha+2}dx=E(u_0),
		\end{align*}
		for all $t\in [0,T)$. 
	\end{itemize}
\end{proposition}
We refer the reader to \cite[Proposition 2.5]{Dinh-mfnls} for the proof of this result.

Now, let us denote
\begin{align}
I(v):= s \|u(t)\|^2_{\dot{H}^s} - \frac{d\alpha}{2(\alpha+2)} \|u(t)\|^{\alpha+2}_{L^{\alpha+2}}. \label{define I}
\end{align}
Note that if we take
\begin{align}
v^\lambda(x):= \lambda^{\frac{d}{2}} v(\lambda x), \label{scaling}
\end{align}
then a simple computation shows
\[
\|v^\lambda\|_{L^2}=\|v\|_{L^2}, \quad \|v^\lambda\|_{\dot{H}^s} = \lambda^s \|v\|_{\dot{H}^s}, \quad \|v^\lambda\|_{L^{\alpha+2}} = \lambda^{\frac{d\alpha}{2(\alpha+2)}} \|v\|_{L^{\alpha+2}}.
\]
We also have
\begin{align*}
S_\omega(v^\lambda)&= \frac{1}{2} \|v^\lambda\|^2_{\dot{H}^s} +\frac{\omega}{2} \|v^\lambda\|^2_{L^2} - \frac{1}{\alpha+2} \|v^\lambda\|^{\alpha+2}_{L^{\alpha+2}} \\
&= \frac{\lambda^{2s}}{2} \|v\|^2_{\dot{H}^s} +\frac{\omega}{2}\|v\|^2_{L^2} - \frac{\lambda^{\frac{d\alpha}{2}}}{\alpha+2} \|v\|^{\alpha+2}_{L^{\alpha+2}}.
\end{align*}
It is easy to see that
\[
I(v) = \left. \partial_\lambda S_\omega(v^\lambda) \right|_{\lambda=1}.
\]
We have the following characterization of the ground state in the mass-supercritical and energy-subcritical case.
\begin{lemma} \label{lemma characterization supercritical}
Let $d\geq 1$, $0<s<1$, $\frac{4s}{d} <\alpha<\alpha^\star$ with $\alpha^\star$ as in $(\ref{define alpha star})$ and $\omega>0$. Let $\phi_\omega$ be the ground state related to $(\ref{elliptic equation omega})$. Then 
\begin{align}
S_\omega(\phi_\omega) =\inf \left\{ S_\omega(v) \ : \ v \in H^s \backslash \{0\}, \ I(v) =0 \right\}. \label{characterization supercritical}
\end{align}
\end{lemma}

\begin{proof}
Denote $m:= \inf \left\{ S_\omega(v) \ : \ v \in H^s \backslash \{0\}, \ I(v) =0 \right\}$. By Pohozaev's identities, it is easy to check that $I(\phi_\omega) = K_\omega(\phi_\omega) =0$. By definition of $m$, we have
\begin{align}
S_\omega(\phi_\omega) \geq m. \label{characterization proof 1}
\end{align}
Let $v \in H^s \backslash \{0\}$ be such that $I(v) =0$. If $K_\omega(v) =0$, then $S_\omega(v) \geq S_\omega(\phi_\omega)$. Assume $K_\omega(v) \ne 0$. We have 
\[
K_\omega (v^\lambda) = \lambda^{2s} \|v\|^2_{\dot{H}^s} + \omega \|v\|^2_{L^2} - \lambda^{\frac{d\alpha}{2}} \|v\|^{\alpha+2}_{L^{\alpha+2}},
\]
where $v^\lambda$ is as in $(\ref{scaling})$. Since $\frac{d\alpha}{2} >2s$, we see that $\lim_{\lambda \rightarrow 0} K_\omega(v^\lambda)= \omega \|v\|^2_{L^2} >0$ and $\lim_{\lambda \rightarrow +\infty} K_\omega(v^\lambda) = -\infty$. It follows that there exists $\lambda_0>0$ such that $K_\omega(v^{\lambda_0}) =0$. By $(\ref{property S_omega phi_omega})$, we get $S_\omega(v^{\lambda_0}) \geq S_\omega(\phi_\omega)$. On the other hand,
\[
\partial_\lambda S_\omega(v^\lambda) = s \lambda^{2s-1} \|v\|^2_{\dot{H}^s} - \frac{d\alpha}{2(\alpha+2)} \lambda^{\frac{d\alpha}{2}-1} \|v\|^{\alpha+2}_{L^{\alpha+2}} = \frac{I(v^\lambda)}{\lambda}.
\]
It is easy to see that the equation $\partial_\lambda S_\omega(v^\lambda)=0$ admits a unique non-zero solution
\[
\left( \frac{\|v\|^2_{\dot{H}^s}}{\frac{d\alpha}{2s(\alpha+2)} \|v\|^{\alpha+2}_{L^{\alpha+2}} } \right)^{\frac{2}{d\alpha-4s}} =1.
\]
The last equality comes from the fact that $I(v) =0$. This implies that 
\[
\left\{
\begin{array}{cl}
\partial_\lambda S_\omega(v^\lambda) >0 & \text{if } \lambda \in (0,1), \\
\partial_\lambda S_\omega(v^\lambda) <0 & \text{if } \lambda \in (1, \infty).
\end{array}
\right.
\]
In particular, we have $S_\omega(v^\lambda) < S_\omega(v)$ for any $\lambda>0, \lambda \ne 1$. Since $\lambda_0>0$, we have $S_\omega(v^{\lambda_0}) \leq S_\omega(v)$. Thus, $S_\omega(\phi_\omega) \leq S_\omega(v)$ for any $v\in H^s \backslash \{0\}$ satisfying $I(v)=0$. It follows that
\begin{align}
S_\omega(\phi_\omega) \leq m. \label{characterization proof 2}
\end{align}
Combining $(\ref{characterization proof 1})$ and $(\ref{characterization proof 2})$, the proof is complete.
\end{proof}

Let $\phi_\omega$ be the ground state related to $(\ref{elliptic equation omega})$. We define 
\[
\mathcal{C}_\omega := \left\{ v \in H^s \backslash \{0\} \ : \ S_\omega(v) < S_\omega(\phi_\omega), \ I(v) <0 \right\}.
\]

\begin{lemma} \label{lemma invariant C_omega}
Let $d\geq 1$, $0<s<1$, $\frac{4s}{d}<\alpha <\alpha^\star$ with $\alpha^\star$ as in $(\ref{define alpha star})$, $\omega>0$ and $\phi_\omega$ be the ground state related to $(\ref{elliptic equation omega})$. Then the set $\mathcal{C}_\omega$ is invariant under the flow of $(\ref{FNLS})$, that is, if $u_0 \in \mathcal{C}_\omega$, then the solution $u(t)$ to $(\ref{FNLS})$ with initial data $u_0$ belongs to $\mathcal{C}_\omega$ for any $t$ in the existence time.
\end{lemma}

\begin{proof}
Let $u_0 \in \mathcal{C}_\omega$. By the conservation of mass and energy, we have
\begin{align}
S_\omega(u(t)) = S_\omega(u_0) <S_\omega(\phi_\omega), \label{estimate S_omega}
\end{align}
for any $t$ in the existence time. It remains to show $I(u(t))<0$ for any $t$ in the existence time. Suppose that there exists $t_0$ such that $I(u(t_0)) \geq 0$. By the continuity of the function $t\mapsto I(u(t))$, there exists $t_1 \in (0,t_0]$ so that $I(u(t_1))=0$. By $(\ref{characterization supercritical})$, we have $S_\omega(u(t_1)) \geq S_\omega(\phi_\omega)$ which contradicts with $(\ref{estimate S_omega})$. The proof is complete. 
\end{proof}

\begin{lemma} \label{lemma key estimate}
Let $d\geq 1$, $0<s<1$, $\frac{4s}{d}<\alpha<\alpha^\star$ with $\alpha^\star$ as in $(\ref{define alpha star})$, $\omega>0$ and $\phi_\omega$ be the ground state related to $(\ref{elliptic equation omega})$. If $v \in \mathcal{C}_\omega$, then
\begin{align}
I(v) \leq 2s(S_\omega(v) - S_\omega(\phi_\omega)). \label{key estimate}
\end{align}
\end{lemma}

\begin{proof}
Denote
\[
f(\lambda) := S_\omega(v^\lambda) = \frac{\lambda^{2s}}{2} \|v\|^2_{\dot{H}^s} +\frac{\omega}{2} \|v\|^2_{L^2} - \frac{\lambda^{\frac{d\alpha}{2}}}{\alpha+2} \|v\|^{\alpha+2}_{L^{\alpha+2}}.
\]
We have
\begin{align}
f'(\lambda)= s \lambda^{2s-1} \|v\|^2_{\dot{H}^s} -\frac{d\alpha}{2(\alpha+2)} \lambda^{\frac{d\alpha}{2}-1} \|v\|^{\alpha+2}_{L^{\alpha+2}} = \frac{I(v^\lambda)}{\lambda}. \label{property f 1}
\end{align}
We also have
\begin{align*}
(\lambda f'(\lambda))' &= 2s^2 \lambda^{2s-1} \|v\|^2_{\dot{H}^s} - \frac{d^2\alpha^2}{4(\alpha+2)} \lambda^{\frac{d\alpha}{2}-1} \|v\|^{\alpha+2}_{L^{\alpha+2}}  \\
&= 2s \left(s \lambda^{2s-1} \|v\|^2_{\dot{H}^s} - \frac{d\alpha}{2(\alpha+2)} \lambda^{\frac{d\alpha}{2}-1} \|v\|^{\alpha+2}_{L^{\alpha+2}} \right) - \frac{d\alpha(d\alpha-4s)}{4(\alpha+2)} \lambda^{\frac{d\alpha}{2}-1} \|v\|^{\alpha+2}_{L^{\alpha+2}}  \\
&=2s f'(\lambda) - \frac{d\alpha(d\alpha-4s)}{4(\alpha+2)} \lambda^{\frac{d\alpha}{2}-1} \|v\|^{\alpha+2}_{L^{\alpha+2}}.
\end{align*}
Since $d\alpha>4s$, we thus get
\begin{align}
(\lambda f'(\lambda))' \leq 2s f'(\lambda), \label{property f 2}
\end{align}
for all $\lambda>0$. Note that since $I(v)<0$, we see that the equation $\partial_\lambda S_\omega(v^\lambda) =0$ admits a unique non-zero solution
\[
\lambda_0 = \left( \frac{\|v\|^2_{\dot{H}^s}}{\frac{d\alpha}{2s(\alpha+2)} \|v\|^{\alpha+2}_{L^{\alpha+2}} } \right)^{\frac{2}{d\alpha-4s}} \in (0,1).
\]
It follows that $I(v^{\lambda_0}) = \lambda_0 \left. \partial_\lambda S_\omega (v^\lambda) \right|_{\lambda=\lambda_0} =0$. Taking integration $(\ref{property f 2})$ over $(\lambda_0,1)$ and using $(\ref{property f 1})$, we obtain
\[
I(v) - I(v^{\lambda_0}) \leq 2s (S_\omega(v) - S_\omega(v^{\lambda_0})) \leq 2s (S_\omega(v) - S_\omega(\phi_\omega)). 
\]
Here the last inequality follows from $(\ref{characterization supercritical})$ and the fact $I(v^{\lambda_0})=0$. The proof is complete.
\end{proof}

We next recall the localized virial estimate related to $(\ref{FNLS})$ which is the main ingredient in the proof of the strong instability of the ground state standing wave. The localized virial estimate was used by Boulenger-Himmelsbach-Lenzmann \cite{BoulengerHimmelsbachLenzmann} to show the existence of finite time blow-up radial solutions to $(\ref{FNLS})$ in the mass-critical and mass-supercritical cases. Let us start with the following estimate.
\begin{lemma} \label{lemma various estimate 1}
	Let $d\geq 1$ and $\varphi: \R^d \rightarrow \R$ be such that $\nabla \varphi \in W^{1,\infty}$. Then for all $u \in H^{1/2}$, 
	\[
	\left| \int \overline{u}(x) \nabla \varphi(x) \cdot \nabla u(x) dx \right| \leq C\left( \||\nabla|^{1/2} u\|^2_{L^2} + \|u\|_{L^2} \||\nabla|^{1/2} u\|_{L^2} \right),
	\]
	for some $C>0$ depending only on $\|\nabla \varphi\|_{W^{1,\infty}}$ and $d$.
\end{lemma}
Now, let $d\geq 1$, $1/2  \leq s <1$ and $\varphi: \R^d \rightarrow \R$ be such that $\nabla \varphi \in W^{3,\infty}$. Assume $u \in C([0,T), H^s)$ is a solution to $(\ref{FNLS})$. Note that in \cite{BoulengerHimmelsbachLenzmann}, Boulenger-Himmelsbach-Lenzmann assume $u \in C([0,T), H^{2s})$ due to the lack of local theory at that time. Thanks to the local theory (see e.g. \cite{Dinh-fract, Dinh-mfnls, HongSire}), one can recover $H^s$-valued solutions by an approximation argument (see \cite[Section 2]{BoulengerHimmelsbachLenzmann}). The localized virial action of $u$ is defined by
\[
M_\varphi(u(t)):= 2 \int \nabla \varphi(x) \cdot \im{(\overline{u}(t,x) \nabla u(t,x))} dx.
\]
We see that $M_\varphi(u(t))$ is well-defined. Indeed, by Lemma $\ref{lemma various estimate 1}$, 
\[
|M_\varphi(u(t))| \lesssim C(\varphi) \|u(t)\|^2_{H^{1/2}} \lesssim C(\varphi) \|u(t)\|^2_{H^s} <\infty.
\]
In order to study the time evolution of $M_\varphi(u(t))$, we need to introduce the following auxiliary function
\begin{align}
u_m(t,x):= c_s \frac{1}{-\Delta +m} u(t,x) = c_s \mathcal{F}^{-1} \left( \frac{\hat{u}(t,\xi)}{|\xi|^2 +m} \right), \quad m>0, \label{auxiliary function}
\end{align}
where
\[
c_s:= \sqrt{\frac{\sin \pi s}{\pi}}
\]
is the normalization factor. Remark that since $u(t) \in H^s$, the smoothing operator $(-\Delta +m)^{-1}$ implies that $u_m(t) \in H^{s+2}$.

We have the following time evolution of $M_\varphi(u(t))$ (see \cite[Lemma 2.1]{BoulengerHimmelsbachLenzmann}).
\begin{lemma} [Time evolution $M_\varphi(u(t))$ \cite{BoulengerHimmelsbachLenzmann}] \label{lemma time evolution}
Let $d\geq 1, 1/2 < s <1$ and $\varphi: \R^d \rightarrow \R$ be such that $\nabla \varphi \in W^{3,\infty}$. Assume that $u \in C([0,T), H^s)$ is a solution to $(\ref{FNLS})$. Then for any $t\in [0,T)$, it holds that
\begin{align*}
\frac{d}{dt} M_\varphi(u(t)) &= - \int_0^\infty m^s \int \Delta^2 \varphi |u_m(t)|^2 dx dm + 4 \sum_{j,k=1}^d \int_0^\infty m^s \int \partial^2_{jk} \varphi \partial_j \overline{u}_m(t) \partial_k u_m(t) dx dm \nonumber \\
&\mathrel{\phantom{= - \int_0^\infty m^s \int \Delta^2 \varphi |u_m(t)|^2 dx dm }} - \frac{2\alpha}{\alpha+2} \int \Delta \varphi |u(t)|^{\alpha+2} dx,
\end{align*}
where $u_m(t)$ is as in $(\ref{auxiliary function})$.
\end{lemma}
\begin{remark} \label{remark time evolution}
Using Plancherel's and Fubini's theorems, it follows that
\begin{align}
\begin{aligned}
\int_0^\infty m^s \int |\nabla u_m| dx dm &=\int \left( \frac{\sin \pi s}{\pi} \int_0^\infty \frac{m^s dm}{(|\xi|^2 + m)^2} \right) |\xi|^2 |\hat{u}(\xi)|^2 d\xi \\
&= \int (s|\xi|^{2s-2})|\xi|^2 |\hat{u}(\xi)|^2 d\xi = s \|u\|^2_{\dot{H}^s}.
\end{aligned}
\label{auxiliary identity}
\end{align}
If we make a formal substitution and take the unbounded function $\varphi(x) = |x|^2$, then by Lemma $\ref{lemma time evolution}$ and $(\ref{auxiliary identity})$, we find formally the virial identity
\begin{align}
\begin{aligned}
\frac{d}{dt} M_{|x|^2} (u(t)) &= 8s \|u(t)\|^2_{\dot{H}^s} - \frac{4d\alpha}{\alpha+2} \|u(t)\|^{\alpha+2}_{L^{\alpha+2}} \\
&=4d\alpha E(u(t))-2(d\alpha-4s) \|u(t)\|^2_{\dot{H}^s} \\
&= 8 I(u(t)),
\end{aligned}
\label{virial identity}
\end{align}
where $I$ is given in $(\ref{define I})$.
\end{remark}

We now recall localized virial estimates for radial $H^s$ solutions related to $(\ref{FNLS})$. Let $\varphi: \R^d \rightarrow \R$ be as above. We assume in addition that $\varphi$ is radially symmetric and satisfies
\[
\varphi(r):= \left\{
\begin{array}{cl}
r^2 &\text{if } r\leq 1, \\
\text{const.} &\text{if } r\geq 10,
\end{array}
\right.
\quad \text{and} \quad \varphi''(r) \leq 2 \text{ for } r\geq 0.
\]
The precise constant here is not important. For $R>1$ given, we define the scaled function $\varphi_R: \R^d \rightarrow \R$ by
\begin{align}
\varphi_R(x) = \varphi_R(r):= R^2 \varphi(r/R), \quad r=|x|. \label{define varphi_R}
\end{align}
It is easy to check that
\[
2-\varphi''_R(r) \geq 0, \quad 2-\frac{\varphi'_R(r)}{r} \geq 0, \quad 2d-\Delta \varphi_R(x) \geq 0, \quad \forall r\geq 0, \forall x \in \R^d.
\]
Using Lemma $\ref{lemma time evolution}$, we have the following localized virial estimate for the time evolution of $M_{\varphi_R}(u(t))$.
\begin{lemma}[Localized virial estimate \cite{BoulengerHimmelsbachLenzmann}] \label{lemma localized virial estimate}
Let $d\geq 2$, $\frac{d}{2d-1} \leq s <1$, $0<\alpha <\frac{4s}{d-2s}$, $\varphi_R$ be as in $(\ref{define varphi_R})$. Let $u \in C([0,T), H^s)$ be a radial solution to $(\ref{FNLS})$. Then for any $t \in [0,T)$,
\begin{align}
\begin{aligned}
\frac{d}{dt} M_{\varphi_R}(u(t)) &\leq 8s \|u(t)\|^2_{\dot{H}^s} - \frac{4d\alpha}{\alpha+2} \|u(t)\|^{\alpha+2}_{L^{\alpha+2}}  + O\left(R^{-2s} + R^{-\frac{\alpha(d-1)}{2} + \eps s} \|u(t)\|^{\frac{\alpha}{2s} +\eps}_{\dot{H}^s} \right) \\
&=4d\alpha E(u(t)) - 2(d\alpha-4s) \|u(t)\|^2_{\dot{H}^s} + O\left(R^{-2s} + R^{-\frac{\alpha(d-1)}{2} + \eps s} \|u(t)\|^{\frac{\alpha}{2s} +\eps}_{\dot{H}^s} \right) \\
&=8I(u(t)) + O\left(R^{-2s} + R^{-\frac{\alpha(d-1)}{2} + \eps s} \|u(t)\|^{\frac{\alpha}{2s} +\eps}_{\dot{H}^s} \right),
\end{aligned}
\label{localized virial estimate}
\end{align}
for any $0<\eps<\frac{(2s-1)\alpha}{2s}$. Here the implicit constant depends only on $\|u_0\|_{L^2}, d, \eps, \alpha$ and $s$.
\end{lemma}
We refer the reader to \cite[Lemma 2.2]{BoulengerHimmelsbachLenzmann} for the proof of the above result. 

\begin{remark} \label{remark localized virial estimate}
	\begin{itemize}
		\item The restriction $\frac{d}{2d-1} \leq s <1$ follows from the local well-posedness of radial $H^s$ solutions for $(\ref{FNLS})$ given in Proposition $\ref{proposition local well-posedness}$.
		\item In practice, we need the exponent $\frac{\alpha}{2s} + \eps$ to be smaller than or equal to 2. This leads to the restriction $\alpha <4s$.
	\end{itemize}
\end{remark}

We are now able to prove our main result-Theorem $\ref{theorem strong instability}$.

\noindent {\it Proof of Theorem $\ref{theorem strong instability}$.}
Let $\eps>0$. Since $\phi^\lambda_\omega \rightarrow \phi_\omega$ in $H^s$ as $\lambda \rightarrow 1$. There exists $\lambda_0>1$ such that $\|\phi_\omega -\phi^{\lambda_0}_\omega\|_{H^s} <\eps$. By decreasing $\lambda_0$ if necessary, we claim that $\phi^{\lambda_0}_\omega \in C_\omega$. Indeed, a direct computation shows that
\[
S_\omega(\phi^\lambda_\omega) = \frac{\lambda^{2s}}{2}\|\phi_\omega\|^2_{\dot{H}^s} +\frac{\omega}{2}\|\phi_\omega\|^2_{L^2} - \frac{\lambda^{\frac{d\alpha}{2}}}{\alpha+2} \|\phi_\omega\|^{\alpha+2}_{L^{\alpha+2}},
\]
and
\[
\partial_\lambda S_\omega(\phi^\lambda_\omega) = s \lambda^{2s-1} \|\phi_\omega\|^2_{\dot{H}^s} - \frac{d\alpha}{2(\alpha+2)} \lambda^{\frac{d\alpha}{2}-1} \|\phi_\omega\|^{\alpha+2}_{L^{\alpha+2}} = \frac{I(\phi^\lambda_\omega)}{\lambda}.
\]
It is not hard to see that the equation $\partial_\lambda S_\omega(\phi^\lambda_\omega) =0$ admits a unique non-zero solution 
\[
\left(\frac{\|\phi_\omega\|^2_{\dot{H}^s}}{\frac{d\alpha}{2s(\alpha+2)} \|\phi_\omega\|^{\alpha+2}_{L^{\alpha+2}} } \right)^{\frac{2}{d\alpha-4s}} =1.
\]
Note that the last equality comes from the fact that $I(\phi_\omega) =0$, which follows from the Pohozaev's identities $(\ref{pohozaev identities elliptic equation omega})$. This implies that
\[
\left\{
\begin{array}{cl}
\partial_\lambda S_\omega(\phi^\lambda_\omega) >0 &\text{if } \lambda \in (0,1), \\
\partial_\lambda S_\omega(\phi^\lambda_\omega) <0 &\text{if } \lambda \in (1,\infty).
\end{array}
\right.
\]
We thus get $S_\omega(\phi^\lambda_\omega) <S_\omega(\phi_\omega)$ for any $\lambda>0, \lambda \ne 1$. Since $I(\phi^\lambda_\omega) = \lambda \partial_\lambda S_\omega(\phi^\lambda_\omega)$, we have
\[
\left\{
\begin{array}{cl}
I(\phi^\lambda_\omega) >0 &\text{if } \lambda \in (0,1), \\
I(\phi^\lambda_\omega) <0 &\text{if } \lambda \in (1,\infty).
\end{array}
\right.
\]
We thus obtain
\[
S_\omega(\phi^{\lambda_0}_\omega) <S_\omega(\phi_\omega), \quad I(\phi^{\lambda_0}) <0.
\]
It follows that $\phi^{\lambda_0}_\omega \in C_\omega$. 

By Proposition $\ref{proposition local well-posedness}$, we see that under the assumption $d\geq 2$, $\frac{d}{2d-1} \leq s<1$ and $\frac{4s}{d}<\alpha<\frac{4s}{d-2s}$, there exists a unique solution $u \in C([0,T), H^s_{\text{rad}})$ with initial data $u_0 = \phi^{\lambda_0}_\omega$, where $T>0$ is the maximal time of existence. Note that $\phi^{\lambda_0}_\omega$ is radially symmetric. We will show that the solution $u$ blows up in finite time. It is done by several steps. 

\noindent {\bf Step 1.} We claim that there exists $a>0$ such that
\begin{align}
I(u(t)) \leq -a, \quad \forall t \in [0,T). \label{negativity}
\end{align}
Indeed, since $C_\omega$ is invariant under the flow of $(\ref{FNLS})$ with $d\geq 2$, $\frac{d}{2d-1} \leq s<1$ and $\frac{4s}{d}<\alpha <\frac{4s}{d-2s}$, we have $u(t) \in C_\omega$ for all $t\in [0,T)$. By Lemma $\ref{lemma key estimate}$, we have
\[
I(u(t)) \leq 2s (S_\omega(u(t)) - S_\omega(\phi_\omega)) = 2s(S_\omega(\phi^{\lambda_0}_\omega) - S_\omega(\phi_\omega)).
\]
This proves $(\ref{negativity})$ with $a = 2s(S_\omega(\phi_\omega) - S_\omega(\phi^{\lambda_0}_\omega))>0$. 

\noindent {\bf Step 2.} We next claim that there exists $b>0$ such that
\begin{align}
\frac{d}{dt} M_{\varphi_R}(u(t)) &\leq -b \|u(t)\|^2_{\dot{H}^s}, \label{estimate 1} \\
\|u(t)\|_{\dot{H}^s} & \gtrsim 1, \label{estimate 2}
\end{align}
for all $t\in [0,T)$, where $\varphi_R$ is as in $(\ref{define varphi_R})$. Let us first prove $(\ref{estimate 2})$. Assume that $(\ref{estimate 2})$ is not true. Then there exists $(t_n)_{n\geq 1}$ a time sequence in $[0,T)$ such that $\|u(t_n)\|_{\dot{H}^s} \rightarrow 0$ as $n\rightarrow \infty$. By the sharp Gagliardo-Nirenberg inequality $(\ref{sharp gagliardo-nirenberg inequality})$, we have
\[
\|u(t_n)\|^{\alpha+2}_{L^{\alpha+2}} \leq C_{\text{opt}} \|u(t_n)\|^{\frac{d\alpha}{2s}}_{\dot{H}^s} \|u(t_n)\|^{\alpha+2-\frac{d\alpha}{2s}}_{L^2} \rightarrow 0,
\]
as $n\rightarrow \infty$. Here we use the conservation of mass to get the last convergence. It follows that 
\[
I(u(t_n)) = s \|u(t_n)\|^2_{\dot{H}^s} - \frac{d\alpha}{2(\alpha+2)} \|u(t_n)\|^{\alpha+2}_{L^{\alpha+2}} \rightarrow 0,
\]
as $n\rightarrow \infty$, which contradicts to $(\ref{negativity})$. We now prove $(\ref{estimate 1})$. Since $u(t)$ is radially symmetric, we apply Lemma $\ref{lemma localized virial estimate}$ to have
\[
\frac{d}{dt} M_{\varphi_R}(u(t)) \leq 4d\alpha E(u(t)) - 2(d\alpha-4s)\|u(t)\|^2_{\dot{H}^s} + O \left( R^{-2s} + R^{-\frac{\alpha(d-1)}{2} +\eps s} \|u(t)\|^{\frac{\alpha}{2s} +\eps}_{\dot{H}^s} \right),
\]
for any $t\in [0,T)$ and any $R>1$. Thanks to the assumption $\alpha <4s$, we can apply the Young inequality to get for any $\eta>0$,
\[
R^{-\frac{\alpha(d-1)}{2} +\eps s} \|u(t)\|^{\frac{\alpha}{2s} +\eps}_{\dot{H}^s} \lesssim \eta \|u(t)\|^2_{\dot{H}^s} + \eta^{-\frac{\alpha+2\eps s}{4s-\alpha - 2\eps s}} R^{-\frac{2s(\alpha(d-1)-2\eps s)}{4s-\alpha -2\eps s}}.
\]
We thus get
\[
\frac{d}{dt} M_{\varphi_R}(u(t)) \leq 4d\alpha E(u(t)) - 2(d\alpha-4s) \|u(t)\|^2_{\dot{H}^s} + C\eta \|u(t)\|^2_{\dot{H}^s} + O \left( R^{-2s} + \eta^{-\frac{\alpha+2\eps s}{4s-\alpha - 2\eps s}} R^{-\frac{2s(\alpha(d-1)-2\eps s)}{4s-\alpha -2\eps s}} \right),
\]
for any $t\in [0,T)$, any $\eta>0$, any $R>1$ and some constant $C>0$. 

Now, we fix $t\in [0,T)$ and denote
\[
\mu:= \frac{4d\alpha |E(u_0)| +2}{d\alpha -4s}.
\]
We consider two cases.

\noindent {\bf Case 1.} 
\[
\|u(t)\|^2_{\dot{H}^s} \leq \mu.
\]
Since $4d\alpha E(u(t)) - 2 (d\alpha -4s) \|u(t)\|^2_{\dot{H}^s} = 8 I(u(t)) \leq -8a$ for all $t\in [0,T)$, we have 
\[
\frac{d}{dt} M_{\varphi_R}(u(t)) \leq -8a + C\eta \mu + O \left( R^{-2s} + \eta^{-\frac{\alpha+2\eps s}{4s-\alpha - 2\eps s}} R^{-\frac{2s(\alpha(d-1)-2\eps s)}{4s-\alpha -2\eps s}} \right).
\]
By choosing $\eta>0$ small enough and $R>1$ large enough depending on $\eta$, we see that
\[
\frac{d}{dt} M_{\varphi_R}(u(t)) \leq -4a \leq -\frac{4a}{\mu} \|u(t)\|^2_{\dot{H}^s}.
\]

\noindent {\bf Case 2.} 
\[
\|u(t)\|^2_{\dot{H}^s} \geq \mu.
\]
In this case, we have
\[
4d\alpha E(u_0) - 2(d\alpha-4s) \|u(t)\|^2_{\dot{H}^s} \leq 4d\alpha E(u_0) - (d\alpha -4s) \mu - (d\alpha-4s)\|u(t)\|^2_{\dot{H}^s} \leq -2 -(d\alpha-4s)\|u(t)\|^2_{\dot{H}^s}.
\]
Thus,
\[
\frac{d}{dt} M_{\varphi_R}(u(t)) \leq -2 - (d\alpha -4s)\|u(t)\|^2_{\dot{H}^s} + C\eta \|u(t)\|^2_{\dot{H}^s} + O\left( R^{-2s} + \eta^{-\frac{\alpha+2\eps s}{4s-\alpha - 2\eps s}} R^{-\frac{2s(\alpha(d-1)-2\eps s)}{4s-\alpha -2\eps s}} \right).
\]
Since $d\alpha-4s>0$, we choose $\eta>0$ small enough so that
\[
d\alpha-4s - C\eta \geq \frac{d\alpha-4s}{2}.
\]
We next choose $R>1$ large enough depending on $\eta$ so that
\[
-2 + O\left( R^{-2s} + \eta^{-\frac{\alpha+2\eps s}{4s-\alpha - 2\eps s}} R^{-\frac{2s(\alpha(d-1)-2\eps s)}{4s-\alpha -2\eps s}} \right) \leq 0.
\]
We thus obtain
\[
\frac{d}{dt} M_{\varphi_R}(u(t)) \leq - \frac{d\alpha-4s}{2} \|u(t)\|^2_{\dot{H}^s}.
\]
In both cases, the choices of $\eta>0$ and $R>1$ are independent of $t$. Therefore, $(\ref{estimate 1})$ follows with $b= \min \left\{\frac{4a}{\mu}, \frac{d\alpha-4s}{2} \right\}>0$.

\noindent {\bf Step 3.} We are now able to show that the solution $u$ blows up in finite time. Assume by contradiction that $T=+\infty$. By $(\ref{estimate 1})$ and $(\ref{estimate 2})$, we see that $\frac{d}{dt} M_{\varphi_R}(u(t)) \leq -C$ for some $C>0$. Integrating this bound, it yields that $M_{\varphi_R}(u(t)) <0$ for all $t\geq t_0$ with some $t_0 \gg 1$ large enough. Taking integration over $[t_0,t]$ of $(\ref{estimate 1})$, we obtain
\begin{align}
M_{\varphi_R}(u(t)) \leq -b \int_{t_0}^t \|u(\tau)\|^2_{\dot{H}^s}d\tau, \label{estimate 3}
\end{align}
for all $t\geq t_0$. On the other hand, by Lemma $\ref{lemma various estimate 1}$ and the conservation of mass, we have
\begin{align}
|M_{\varphi_R}(u(t))| \leq C(\varphi_R) \left(\|u(t)\|^{\frac{1}{s}}_{\dot{H}^s} + \|u(t)\|^{\frac{1}{2s}}_{\dot{H}^s} \right), \label{estimate 4}
\end{align}
where we have used the interpolation estimate $\|u\|_{\dot{H}^{\frac{1}{2}}} \lesssim \|u\|_{L^2}^{1-\frac{1}{2s}} \|u\|^{\frac{1}{2s}}_{\dot{H}^s}$.
Combining $(\ref{estimate 2})$ and $(\ref{estimate 4})$, we get
\begin{align}
|M_{\varphi_R}(u(t))| \leq C(\varphi_R) \|u(t)\|^{\frac{1}{s}}_{\dot{H}^s}. \label{estimate 5}
\end{align}
It follows from $(\ref{estimate 3})$ and $(\ref{estimate 5})$ that 
\begin{align}
M_{\varphi_R}(u(t)) \leq -A \int_{t_0}^t |M_{\varphi_R}(u(\tau))|^{2s} d\tau, \label{integral inequality}
\end{align}
for all $t\geq t_0$ with some constant $A=C(b,R)>0$. Set $z(t):= \int_{t_0}^t |M_{\varphi_R}(u(\tau))|^{2s} d\tau$ for $t\geq t_0$ and fix some time $t_1>t_0$. We see that $z(t)$ is strictly increasing and non-negative. Moreover,
\[
z'(t) = |M_{\varphi_R}(u(t))|^{2s} \geq A^{2s} z^{2s}(t), 
\]
for some $A=C(b,R)>0$. Integrating the above inequality on $[t_1,t]$, we get
\[
z(t) \geq \frac{z(t_1)}{\left(1-(2s-1)A^{2s} [z(t_1)]^{2s-1} (t-t_1)\right)^{\frac{1}{2s-1}}},
\]
for all $t\geq t_1$. It follows that
\[
z(t) \rightarrow +\infty  \text{ as } t \uparrow t_*:= t_1 + \frac{1}{(2s-1)A^{2s} [z(t_1)]^{2s-1}} >t_1.
\]
By $(\ref{integral inequality})$, we obtain
\[
M_{\varphi_R}(u(t)) \leq -A z(t) \rightarrow -\infty \text{ as } t\uparrow t_*.
\]
Therefore, the solution cannot exist for all time $t\geq 0$ and consequencely we must have $T<+\infty$. The proof is complete.
\defendproof

\section*{Acknowledgments}
The author would like to express his deep gratitude to his wife - Uyen Cong for her encouragement and support. He would like to thank his supervisor Prof. Jean-Marc Bouclet for the kind guidance and constant encouragement. He also would like to thank the reviewer for his/her helpful comments and suggestions. 


\end{document}